\documentclass{article}
\usepackage{amssymb}
\usepackage{amsmath}
\usepackage{amsthm}
\usepackage{MnSymbol}
\usepackage{tikz-cd}
\usepackage{diagrams}
\usepackage{amscd}
\usepackage{graphicx}
\usepackage{mathtools}

\newtheorem{theorem}{Theorem}[section]
\newtheorem{lemma}[theorem]{Lemma}

\theoremstyle{definition}
\newtheorem{definition}[theorem]{Definition}

\theoremstyle{remark}
\newtheorem{remark}[theorem]{Remark}

\numberwithin{equation}{section}

\begin{document}

\title{Unmating of expanding Thurston maps with Julia set $\mathbb{S}^2$} 
\author{Mary Wilkerson}
\date{}

\maketitle
\begin{abstract}

Every expanding Thurston map $f$ without periodic critical points is known to have an iterate $f^n$ which is the topological mating of two polynomials. This has been examined by Kameyama and Meyer; the latter who has offered an explicit construction for finding two polynomials in the unmating of the iterate. Initializing this algorithm depends on an invariant Jordan curve through the postcritical set of $f$--but we propose adjustments to this unmating algorithm for the case where there exists a curve which is fully $f$-invariant up to homotopy and not necessarily simple. When $f$ is a critically pre-periodic expanding Thurston map, extending the algorithm to accommodate non-Jordan curves in this manner allows us to unmate without iterates.

\end{abstract}

\tableofcontents
\addtocontents{toc}{\vskip-40pt} % this is to reduce the gap for the first table of contents entry
 
\let\thefootnote\relax\footnote{\noindent\emph{* Key words and phrases.} mating, unmating, Thurston maps, rational maps.}
\let\thefootnote\relax\footnote{* 2020 Mathematics subject classification: Primary 37F20; Secondary 37F10}
\let\thefootnote\relax\footnote{* This work has been submitted to Contemporary Mathematics. Copyright in this Work may be transferred without further notice.}

 %%%%%%%%%%%%%%%%
%SECTION 1: INTRODUCTION                        
%%%%%%%%%%%%%%%%%

\section{Introduction}
 
Some rational maps contain dynamics similar to those of polynomials. In the 1980s, Douady and Hubbard explored this phenomenon with the introduction of \emph{mating}--which is now the name for a collection of operations that combine two polynomials into a new map with shared dynamics. The classic approach starts with two monic degree $d$ postcritically finite polynomials acting on disjoint copies of their respective filled Julia sets: if appropriate conditions are satisfied, the domains are identified along their Julia set boundaries, yielding a topological sphere. The polynomial mappings then descend to a branched covering of this sphere called the \emph{topological mating}, which is Thurston-equivalent to a rational map on $\widehat{\mathbb{C}}$. 

In the event that the two polynomials are critically preperiodic, their Julia sets are dendrites. This yields that there is an equator-like curve containing a full set of equivalence class representatives for the equivalence relation $\approx$ generating this quotient space. An immediate consequence is that the topological mating in this case is then semiconjugate to the $d$-fold map on $\mathbb{S}^1$. 

Closed curves on $\mathbb{S}^2$ with similar properties to this ``equator" can be very useful in investigating the construction and decomposition of certain rational maps: namely those exhibiting behavior similar to polynomials within their dynamics.   For example, \cite{BOYD2012MEDUSA} and \cite{WILKERSON2019THURSTON} both utilize the Thurston pullback algorithm along with equator-like structures to develop approximations to the associated rational map.  %FIX the S1-PARAM TEXT IN THE PARAGRAPH BELOW. HAS WEIRD MARGIN BEHAVIOR. We're also revoking reference to S1 parameterization, so you may wish to omit this as well?
Conversely, curves with properties resembling those of an equator can be used to decompose or ``unmate" branched coverings arising as matings. In \cite{KAMEYAMA2003PART2}, Kameyama demonstrates that any expanding Thurston maps which is \emph{orientedly $S1$-parametrizable} (See Subsection \ref{oriented}) has an iterate which is Thurston equivalent to an essential mating. In \cite{MEYER2013INVARIANT}, Meyer offers the explicit construction of an $S1$-parameterization. This is elaborated upon in the subsequent papers \cite{MEYER2009EXPANDING} and \cite{MEYER2014UNMATING}, where the $S1$-parameterization is used to unmate a rational map.

Meyer notes that a \emph{pseudo-equator} is a sufficient condition for an expanding postcritically finite rational map with Julia set the Riemann sphere to be a mating, but that it is not necessary: some such rational maps which arise as matings do not have a pseudo-equator. \cite{MEYER2009EXPANDING}. Further, the unmating algorithm described in \cite{MEYER2014UNMATING} does not apply when an initial curve is not Jordan. Our goal is to explore and address these issues. 
 
We start with prerequisite topics in Section \ref{prereq}.  This section includes preliminaries on expanding Thurston maps and matings to discuss the objects we wish to decompose; as well as laminations, oriented curves, and tilings to highlight  tools that will be used to develop their decomposition. We detail the current literature on decomposition of rational maps and unmating in Section \ref{litreview}.  The main results are developed in Section \ref{revised}, where we discuss accommodations to extend the reach of Meyer's unmating algorithm. As the arguments presented here are intended to build upon the constructions given in \cite{KAMEYAMA2003PART2}, \cite{MEYER2009EXPANDING}, \cite{MEYER2013INVARIANT}, and \cite{MEYER2014UNMATING}, we will use corresponding results that are still applicable freely, but also emphasize required adjustments where necessary and highlight implications of our changes.  We then conclude with an example and further avenues of exploration.

%%%%%%%%%%%%%%%%%%%%%%%%%%%%%%%%%%%%%%%%%%%%%                                                                                                                                                                                                                                                           
%
%PREREQUISITES: EXPANDING THURSTON MAPS		
%
%%%%%%%%%%%%%%%%%%%%%%%%%%%%%%%%%%%%%%%%%%%%%%
\section{Preliminaries}\label{prereq}

\subsection{Notation} We utilize many notational conventions adopted in \cite{KAMEYAMA2003PART2}, \cite{MEYER2009EXPANDING}, \cite{MEYER2013INVARIANT}, and \cite{MEYER2014UNMATING} for ease of reference.  Some items are highlighted here. 

We will use $\mathbb{S}^2$ to refer to the two-sphere, but $\widehat{\mathbb{C}}$ to refer to the two-sphere when endowed with a Riemannian metric and structure. 

We occasionally identify $\mathbb{S}^1$ with $\partial{\mathbb{D}}=\{z\in\mathbb{C}:|z|=1\}$ where convenient. Most frequently $\mathbb{S}^1$ will be used to parametrize closed curves however, and our preference in these instances will be to identify $\mathbb{R}/\mathbb{Z}$ with $\mathbb{S}^1$ via the map $t\mapsto e^{2\pi it}$ and use parameters in the interval $[0,1)$. The ``$d$-fold map on $\mathbb{S}^1$ " is then taken to be the map $q_d(t):=dt$ mod 1 applied to this interval. 

When the word \emph{equator} is used in this text, it is taken to represent a curve separating hemispheres of $\mathbb{S}^2$. (This is mentioned as a point of disambiguation, as this word has a distinct mathematical interpretation in \cite{MEYER2014UNMATING}.)

The primary setting of the paper is where $f$ is taken to be an expanding Thurston map which is critically preperiodic. We will denote the degree of $f$ as $d$ and the postcritical set of $f$ as post, post$(f)$, or possibly $\textbf{V}^0$, depending on context.

If a homotopy $H:X\times[0,1]\rightarrow Y$ is an isotopy on $X\times[0,1)$, we will refer to $H$ as a pseudo-isotopy. We may use $H_t$ to refer to the function $H(\cdot,t)$, and generally construct $H_0$ as an identity map. If $H_t$ is constant on a set, we will refer to $H$ as a homotopy \emph{relative to} that set, and say that the homotopy is supported on its complement. 

We use $\bigvee$ to indicate the join of two equivalence relations.

\subsection{Expanding Thurston maps}\label{expanding}

We begin with some preliminary definitions.

\begin{definition} \emph{Thurston maps} are postcritically finite branched coverings of $\mathbb{S}^2$.    If $f,g: \mathbb{S}^2\rightarrow \mathbb{S}^2$ are two orientation-preserving branched coverings with postcritical sets $\text{post}(f)$ and $\text{post}(g)$, we say that $f$ and $g$ are \emph{Thurston equivalent} if and only if there exist orientation-preserving homeomorphisms $h, h': (\mathbb{S}^2,\text{post}(f))\rightarrow(\mathbb{S}^2,\text{post}(g))$ such that $g\circ h' = h\circ f$ and h is isotopic to $h'$ relative to $\text{post}(f)$. 
\end{definition}

In particular, polynomials and rational maps on $\widehat{\mathbb{C}}$ that are postcritically finite are Thurston maps. The Thurston maps we will focus on are rational maps with no periodic critical points. Such rational maps have Julia set equal to $\widehat{\mathbb{C}}$, and are also \emph{expanding} on $\mathbb{S}^2$ as described in \cite{BONK2017EXPANDING}:

\begin{definition}\label{ntiles} Suppose that $f$ is a Thurston map with critical set post:$=\text{post}(f)$, and that $\mathcal{C}\subseteq\mathbb{S}^2$ is a Jordan curve containing post. Note that we may generate a cellular decomposition of $\mathbb{S}^2$ by marking post as our 0-cell vertices, the arcs between these on $\mathcal{C}$ as 1-cell open edges, and the two components of $\mathbb{S}^2-\mathcal{C}$ as 2-cell open tiles. Further, we denote elements of $f^{-n}(post)$ as \emph{$n$-vertices}, the closures of components of $f^{-n}(\mathcal{C}-post)$ as \emph{$n$-edges}, closures of components of $f^{-n}(\mathbb{S}^2-\mathcal{C})$ as \emph{$n$-tiles}. 

We say that $f$ is \emph{expanding} if as $n\rightarrow \infty$ the n-tile diameters shrink to 0 uniformly with respect to some metric generating the topology on $\mathbb{S}^2$. \end{definition}

While the above condition is independent of the metric, much of our discussion hinges on usage of the \emph{visual metric}:

\begin{definition}: Let $f : \mathbb{S}^2\rightarrow\mathbb{S}^2$ be an expanding Thurston map, $\mathcal{C}\subset\mathbb{S}^2$ be a Jordan curve containing post, and $x,y\in\mathbb{S}^2$. If $x\neq y$, we assign $m(x,y)$ to be the largest $n\in\mathbb{N}$ for which there exist intersecting $n$-tiles $X$ and $Y$ with $x\in X$ and $y\in Y$. Otherwise, we set $m(x,y):=\infty$.

A metric $\varrho$ on $\mathbb{S}^2$ is then called a \emph{visual metric} (for f) if there exists a constant $\Lambda > 1$ such that
\begin{equation}\label{diffcurves}\displaystyle\frac{1}{C}\varrho(x, y)\leq \Lambda^{-m(x,y)}\leq C\varrho(x, y)\end{equation}
for all $x,y\in \mathbb{S}^2$, where the constant $C$ is independent of both $x$ and $y$. We call this $\Lambda$ the \emph{expansion factor} of $\varrho$. \cite{BONK2017EXPANDING}
\end{definition}
 
Bonk and Meyer note that visual metrics are guaranteed to exist for expanding Thurston maps, and that each visual metric induces the given topology on $\mathbb{S}^2$. Further, the $\mathcal{C}$ that we start with is not important: if we develop a different $\tilde{\mathcal{C}}$ containing post, we may obtain a similar statement as in \ref{diffcurves}, except with different $\tilde{m}$ and $\tilde{C}$. The expansion factor $\Lambda$, however, will remain the same for this curve. \cite[Prop. 8.3]{BONK2017EXPANDING}

Sets composed of $n$-tiles then have a natural relationship with balls in the visual metric:

\begin{lemma}\label{open}Let $f$ be an expanding Thurston map, and $\mathcal{C}\subset\mathbb{S}^2$ be a Jordan curve containing post. Let $\varrho$ be a visual metric for $f$ corresponding to this choice of $\mathcal{C}$ with expansion factor $\Lambda > 1$. For $x\in\mathbb{S}^2 ,n\in N_0$, set

$U^n(x) := \bigcup\{Y \in \textbf{X}^n :$ Y intersects an $n$-tile $X$ containing $x \}$.

Then there are constants $K \geq 1$ and $n_0 \in \mathbb{N}_0$ with the following property.

For all $x\in\mathbb{S}^2$ and all $n\in \mathbb{Z}$, $B_\varrho(x, r/K)\subset U^n(x)\subset B_\varrho(x, Kr)$,
where $r = \Lambda^{-n}$. \cite[Lemma 8.10]{BONK2017EXPANDING}
\end{lemma}

%%%%%%%

\subsection{Matings}\label{matings}
The work in this paper will emphasize matings of monic polynomials which are critically preperiodic in $\mathcal{C}$. Any such polynomial has an associated Julia set which is a connected and locally connected dendrite--that is, the filled Julia set $K$ has no interior and is the same as the Julia set $J$ for the polynomial.
 
When $K$ is connected, there exists a unique conformal isomorphism $\phi:\hat{\mathbb{C}}- \overline{\mathbb{D}}\rightarrow\hat{\mathbb{C}}- K$ for which $\phi (z^d) = f\circ \phi(z)$. We then let $R(t)$ denote the set of points forming \emph{external ray of angle $t$}; where $R(t):=\{\phi(re^{2\pi i t})\ |\ r\in(1,\infty)\}$. Since any filled Julia set $K$ we discuss here will be locally connected,  $\phi$ extends continuously to $\partial\mathbb{D}$ and thus external rays of angle $t$ are said to have a \emph{landing point}, which is given by $\gamma(t)=\displaystyle\lim_{r\rightarrow 1^+}\phi(re^{2\pi i t})$. The map $\sigma$ is called the \emph{Carath\'{e}odory semiconjugacy}, as we have that $\sigma(dt)=f(\sigma(t))$. In essence, this semiconjugacy relates the mapping behavior of $f$ on its Julia set to that of the $d$-fold map on the circle.

Let $\widetilde{\mathbb{C}}$ be the compactification of $\mathbb{C}$ formed by union with the circle at infinity, $\widetilde{\mathbb{C}}=\mathbb{C}\cup\{\infty\cdot e^{2\pi i \theta}|\theta \in \mathbb{R}/\mathbb{Z}\}$. Then, we may take the \emph{closed external ray of angle $t$}, $\overline{R(t)}$, to be the closure of $R(t)$ in $\widetilde{\mathbb{C}}$. This closure is formed by including the landing point on the Julia set, $\gamma(t)$, and the limit point on the circle at infinity, $\displaystyle\lim_{r\rightarrow \infty}\phi(re^{2\pi i t})$.

We may now introduce three fundamental constructions: the formal, topological, and essential matings.

\begin{definition}Let $P_1:\widetilde{\mathbb{C}}_1\rightarrow\widetilde{\mathbb{C}}_1$ and $P_2:\widetilde{\mathbb{C}}_2\rightarrow\widetilde{\mathbb{C}}_2$ be postcritically finite monic degree $d$ polynomials taken on two disjoint copies of $\widetilde{\mathbb{C}}$, and let $\sim_f$ be the equivalence relation which identifies $\infty \cdot e^{2\pi i t}$ on $\widetilde{\mathbb{C}}_1$ with $\infty \cdot e^{-2\pi i t}$ on $\widetilde{\mathbb{C}}_2$ for all $t\in \mathbb{R}/\mathbb{Z}$. Then, the quotient space $\widetilde{\mathbb{C}}_1\bigsqcup \widetilde{\mathbb{C}}_2/\sim_f$ may be identified with $\mathbb{S}^2$. The map which descends to this quotient space is the \emph{formal mating} of $P_1$ and $P_2$. \end{definition}
 
It should be noted that the formal mating is well-defined and yields a continuous branched covering of $\mathbb{S}^2$ to itself.
 
\begin{definition} Let $\mathbb{S}^2$ denote the domain of the formal mating as above, and let $\sim_t$ be an equivalence relation on $\mathbb{S}^2$ where any points sharing a closed external ray are contained in the same equivalence class . Allowing the formal mating to descend to the quotient space $\mathbb{S}^2/\sim_t$ yields the \emph{topological mating} $P_1\upmodels P_2$.
\end{definition}

Note that a given ray equivalence class of $\sim_t$ may contain points from several external rays because pairs of closed external rays meet where identified by $\sim_f$, and also because multiple external rays may land at multiply accessible points of $J_1$ or $J_2$. This equivalence relation generates a quotient space that can alternately be viewed as resulting from gluing the filled Julia sets of $P_1$ and $P_2$ together along the boundaries. Since the polynomials we examine are critically preperiodic and possess (filled) Julia sets that are dendrites, $\mathbb{S}^2/\sim_t$ has the potential to be a very peculiar space. However, by Moore's theorem if no equivalence class of $\sim_t$ separates $\mathbb{S}^2$, then $\mathbb(S)^2/\sim_t$ is homeomorphic to $\mathbb{S}^2$. %cite?

%INCLUDE DIFFERENCE B/T ARISING AS A MATING AND BEING EQUIVALENT TO A MATING? Definition 3.1. A Thurston map f : S2 → S2 is said to arise as mating if f is topologically conjugate to Pw ⊥ Pb, the (topological) mating of polynomials Pw, Pb. The Thurston map g : S2 → S2 is equivalent to a mating if g is Thurston equivalent to a Thurston map f : S2 → S2 arising as a mating.
 
Moore's theorem is a powerful result, but it may be difficult to visualize the mapping behavior of the topological mating on $\mathbb{S}^2$. We will frequently make use of the \emph{essential} (or \emph{degenerate}) \emph{mating} $P_1\upmodels_eP_2$ instead, as it is somewhat of an intermediary between the formal and topological mating. The intuition behind its development in \cite{LEI1992MATINGS} is this: occasionally, the formal mating is not Thurston equivalent to a rational map due to Levy cycles which yield Thurston obstructions.  In essence, we collapse only the equivalence classes of $\sim_t$ which are needed to fix this problem, and make minor adjustments so that the map descending to the quotient remains a branched covering, as described below.
%(We'll use a FSR construction as in (CITE SELF) to assist in visualizations for this paper.) 

  %FIX THIS SHIT!
\begin{definition}\label{essential} Allow $h:\mathbb{S}^2\rightarrow\mathbb{S}^2$ to denote the formal mating of $P_1$ and $P_2$, and suppose that no point on the critical orbit of $h$ is in an equivalence class of $\sim_t$ which forms a closed loop.

Let $\{\tau_1,...,\tau_m\}$ be the set of equivalence classes of $\sim_t$ that (a) contain at least one point on the critical orbit of $h$ and (b) have an equivalence class of $\sim_t$ containing $\geq 2$ elements of $P_h$ in their forward orbit under $h$. for some n. Define $\sim_e$ as the refinement of $\sim_t$ whose only nontrivial equivalence classes are $\{\tau_1,...,\tau_m\}$. 

Let $\{V_1,...,V_n\}$ be a collection of open tubular neighborhoods (respectively) of $\{\tau_1,...,\tau_n\}$, selected so that distinct $V_i$ are disjoint and that each $V_i$ contains precisely the points on the critical orbit that $\tau_i$ does. Then, define $\{U_1,...,U_m\}$ to be the collection of connected components of $h^{-1}\displaystyle(\bigcup_{i=1}^n V_i)$ which do not intersect any $\tau_i$. 

Let $\pi: \mathbb{S}^2\rightarrow\mathbb{S}^2/\sim_e be$ the natural projection. We define $g: \mathbb{S}^2/\sim_e\rightarrow\mathbb{S}^2/\sim_e$ as follows: 
On each $\overline{U_i}$, define $g$ to be a local homeomorphism which is equal to $\pi \circ h\circ \pi^{-1}$ on $\partial U_i$. Set $g:= \pi \circ h\circ \pi^{-1}$ on the remainder of $\mathbb{S}^2/\sim_e$ . The map $g$ is the \emph{essential mating} of $P_1$ and $P_2$.
\end{definition}
 
Note that since none of the $\tau_i$ are loops, $\mathbb{S}^2/\sim_e$ is homeomorphic to $\mathbb{S}^2$. (We will thus simplify notation by treating $g$ as a self-map of $\mathbb{S}^2$.) As $h$ maps equivalence classes to equivalence classes,  $\pi\circ h\circ\pi^{-1}$ is well-defined on $\mathbb{S}^2$. With the alteration of this composition on each $U_i$ though, we ensure that $g$ does not map arcs to points. The result is a map which is Thurston equivalent to the topological and geometric matings--while retaining much of the relative simplicity in structure of the formal mating. In fact, when $\sim_t$ never identifies post-critical points of the formal mating, the formal and essential matings are the same map. 

%%%%%%%%%%%%%%%%%
\subsection{Laminations}

Recall that a set of the form $\widetilde{\mathbb{C}}_1\bigsqcup \widetilde{\mathbb{C}}_2/\sim_f \ \cong \mathbb{S}^2$ served as the domain of the formal mating. The canonical equator of this space contains a full set of equivalence class representatives of $\sim_t$. Further, restricting to this set, the topological mating is semiconjugate to the $d$-fold map on $\mathbb{S}^1$. As equivalence relations like this will be crucial to our discussion, we introduce a common tool for geometric representation of equivalence relations on $\mathbb{S}^1$: the geodesic lamination.

\begin{definition}Let $L$ be a set of geodesics (which we will call \emph{ leaves}) on the closed hyperbolic disk $\overline{\mathbb{D}}$ which are disjoint, except possibly at their endpoints. If $\bigcup L$ is closed, we call $L$ a \emph{lamination} on $\overline{\mathbb{D}}$. 

The closures of components of $\mathbb{D}-\bigcup L$ are called \emph{gaps}. If a gap has finite intersection with $\partial\mathbb{D}$, this gap is an \emph{ideal polygon}. 
\end{definition}

Laminations were introduced by Thurston as a tool for discussing the geometry and dynamics of polynomials, although 2-sided analogues of laminations have also been used to discuss matings and other maps \cite{THURSTON2009FAMILY} \cite{dastjerdi1991dynamics} \cite{TIMORIN2008EXTERNAL}. To draw a parallel to previous ideas, consider the Julia set $J$ of a postcritically finite polynomial of degree $d$: If any external rays land  at a multiply accessible point of $J$, we mark the corresponding angles for these rays on $\partial\mathbb{D}$, and take as traditional leaves the boundary of the convex hull between the points. This collection of leaves defines a lamination which, when collapsed along its leaves and filled-in ideal polygons, yields a topological model of the structure of $J$. 

\begin{remark}
While we will not reference them heavily in this manuscript, it should be noted that geodesic laminations are referenced heavily in the discussion of equivalence relations in \cite{MEYER2009EXPANDING}. For the convenience of the reader following along, we note that in that discussion, the text uses a somewhat nonstandard definition for ``leaves": these are taken to be the convex hull of a finite collection of points on the boundary of $\mathbb{D}$. 
\end{remark}
 
%%%%%%%%%%%%%%%%
\subsection{Oriented curves and 2-color tilings}\label{oriented}
%%%%%%%%%%%%%%%%%%

Like laminations, tilings are another structure frequently used to represent matings and their approximations. In particular, we will examine 2-color tilings that are associated with oriented curves on $\mathbb{S}^2$.

\begin{definition} We say that a closed curve $\mathcal{C}:\mathbb{S}^1\rightarrow\mathbb{S}^2$ with at most finitely many self-intersections is \emph{oriented} if $\mathcal{C}$ can be deformed to a simple closed curve  by a small perturbation--that is, there exists a continuous map $h:\mathbb{S}^1\times[0,1]\rightarrow\mathbb{S}^2$ such that $h(\mathbb{S}^1,0)=\mathcal{C}$, and $h(\theta,t)\neq h(\theta',t)$ whenever $\theta\neq\theta'$. The key implication here is that while such a curve may self-intersect, it is not allowed to ``self-cross." 

Kameyama refers to such distorted Jordan curves $h(\mathbb{S}^1,t), t\neq 0$ as \emph{unlacings} of $\mathcal{C}$. These are also analagous to Meyer's concept of \emph{geometric representation} as in \cite[Sec. 6]{MEYER2013INVARIANT}.

If desired, unlacings can be constructed so as to be pseudo-isotopic to the original $\mathcal{C}$ \cite[Defn 6.7]{KAMEYAMA2003PART2}. 
\end{definition}

A subtle point of note is that while $\mathcal{C}$ may have an \emph{orientation} implied by its parameterization, the distinction of being \emph{oriented} as above is meant to emphasize how arcs of $\mathcal{C}$ can be viewed as oriented with respect to the components of $\mathbb{S}^2 - \mathcal{C}$. We may view an oriented curve $\mathcal{C}$ as generating a cellular decomposition of $\mathbb{S}^2$ by marking points of self intersection (and possibly finitely many other point along $\mathcal{C}$)  as 0-cell vertices, arcs of $\mathcal{C}$ between these points as 1-cell edges, and the connected components of $\mathbb{S}^2-\mathcal{C}$ as 2-cell tiles.  If these marked points are elements of post for some Thurston map $f$, this is similar to the cellular decomposition discussed in Definition \ref{ntiles}--except here there are more than two tiles. We will denote the sets of closed vertices, edges, and tiles respectively as the set of \emph{$0$-vertices} $\textbf{V}^0$, set of \emph{$0$-edges} $\textbf{E}^0$, and set of \emph{$0$-tiles} $\textbf{X}^0$. (We similarly refer to the sets of pullbacks of these objects by $f^n$ respectively as $\textbf{V}^n$, $\textbf{E}^n$, and $\textbf{X}^n$. In other words, the $n$ in this context is meant to specify an iterate rather than a dimension.) 

If $\mathcal{C}$ is an oriented curve, then the 1-cell arcs surrounding a given tile are either all positively oriented to the tile (in which case we will color it white), or all negatively oriented to the tile (in which case we will color it black). There are no other options for tiles: note that an unlacing of $\mathcal{C}$ is colored similarly, the pseudo-isotopy deforming the unlacing to $\mathcal{C}$ induces this ``checkerboard” coloring, where every edge separates a white tile from a black one.

Any time that we refer to a tiling of $\mathbb{S}^2$ associated with an oriented curve $\mathcal{C}$, we will assume that the tiling was constructed in the above manner. Given such a tiling, we will define connections at vertices of the tiling as follows:

\begin{definition}\label{mark}Let $v$ be any self intersection point of $\mathcal{C}$, $\mathcal{C}'$ an unlacing of $\mathcal{C}$ for which $h:\mathbb{S}^2\times[0,1]\rightarrow\mathbb{S}^2$ is a pseudo-isotopy. Further, let $U(v)$ be some small neighborhood containing $v$ and each other point $u\in\mathcal{C}'$ for which $h(u,1)=v$. (We assume however that no other marked points of $\mathcal{C}$ are contained in this set.) If the deformation of any connected component of $U(v)-\mathcal{C}'$ via $h$ has multiple connected components, the (same-color) tiles containing these components are said to be \emph{connected at $v$}. 
\end{definition}

The intuition is to make note of which tiles an unlacing ``connects" at the vertex $v$. With additional labeling of structures, this satisfies the conditions of a connection per \cite{MEYER2013INVARIANT}.                                                                                                                                                                                                                                                                                                                                                                                                                                                                                                                                                                                               

%%%%%%%%%%%%%%%%%%%%%%%%%%%%%%%%%%%%%%%%                                                                                                                                                                                                                                                           %															%
%                                                             SECTION 5: Unmating 			 %
%                                                             (S1 param and pseudo-equators)	    	%                                                                                                                                                                                                                                                           %															%
%%%%%%%%%%%%%%%%%%%%%%%%%%%%%%%%%%%%%%%%                                                                                                                                                                                                                                                                                                                                                                                                                                                                                                                                                                                                                        
\section{Unmating}\label{litreview}

We have previously noted that the topological mating is conjugate to the $d$-fold map on $\mathbb{S}^1$. Much of the decomposition theory of matings centers on whether such a semiconjugacy--or at least a curve with similar behavior under the $d$-fold map--exists for a given map $f$. We reproduce a number of definitions from \cite{KAMEYAMA2003PART2} to clarify the properties of this curve which are desired, and to give context for later results on the decomposition of expanding Thurston maps via algorithms detailed by Meyer.

\subsection{Fundamental results}

\begin{definition} Let $f$ be an expanding degree $d$ Thurston map with postcritical set post, and let $P_f^a$ denote the set of postcritical points which eventually map forward to a critically periodic cycle for $f$.  Recall $q_d:\mathbb{S}^1\rightarrow\mathbb{S}^1$ denotes the $d$-fold map given by $q_d(t)=td$ mod1 on the circle $\mathbb{S}^1\cong\mathbb{R}/mathbb{Z}\cong[0,1)$.

We say that  $J$ is \emph{S1-parametrizable} if there exists a continuous surjection $\phi:\mathbb{S}^1\rightarrow J$ so that $f\circ\phi=\phi \circ q_N$ for some $N\in\mathbb{N}$.

%do i want commutative diagram for the above?

A closed curve $ \mathcal{C}: \mathbb{S}^1\rightarrow \mathbb{S}^2- P_f^a$ is said to be \emph{fully $f$-invariant up to homotopy} if there exists a closed curve $ \mathcal{C}_1: \mathbb{S}^1\rightarrow  f^{-1}(\mathcal{C})$ and homotopy $ h : \mathbb{S}^1 \times [0,1]\rightarrow \mathbb{S}^2- P_f^a$  from $\mathcal{C}$ to $\mathcal{C}_1$  with $(\text{post}- P_f^a)$ fixed so that $f\circ \mathcal{C}_1=\mathcal{C}\circ q_d$. In other words, we should be able to continuously deform $\mathcal{C}$ to its pullback relative to the postcritical set.
\end{definition}

We now summarize the key result of \cite{KAMEYAMA2003PART2} in the following theorem: 

\begin{theorem}[Kameyama] Suppose $f$ is an expanding Thurston map. If $\mathcal{C}:\mathbb{S}^1\rightarrow\mathbb{S}^2-P_f^a$ is an oriented closed curve which is  fully $f$-invariant up to homotopy, then $J$ is orientedly S1-parametrizable and $f^n$ is Thurston equivalent to the degenerate mating of topological polynomials for some $n \in \mathbb{N}$. 
\end{theorem}

In this paper we emphasize the case where $f$ has only strictly preperiodic critical points and so as a result $P_f^a=\varnothing$ and $J$ is the sphere. In a series of papers  (\cite{MEYER2013INVARIANT},\cite{MEYER2009EXPANDING},\cite{MEYER2014UNMATING}), Meyer adds substantial detail in this setting by constructing a Jordan curve which is fully $f^n$-invariant up to homotopy, giving explicit instructions for using this curve to obtain S1-parameterization, and also demonstrating how to use the S1-parameterization to find labels for the polynomials in the unmating. We briefly summarize the unmating approach below, giving key sections for relevant discussion of each ``step":

\begin{enumerate}%\addtocounter{enumi}{-1}
\item Suppose that $f:\mathbb{S}^2\rightarrow \mathbb{S}^2$ is an expanding Thurston map without periodic critical points. A Jordan curve $\mathcal{C}\subseteq \mathbb{S}^2$ containing post is selected. For a sufficiently large $n\in\mathbb{N}$, there exists a Jordan curve $\mathcal{\gamma}^0$ which is invariant for $F=f^n$ and isotopic to $\mathcal{C}$ relative to post, thus $F$ can be viewed as the subdivision map for a 2-tile subdivision rule on $\mathbb{S}^2$. \cite[Sec. 15]{BONK2017EXPANDING},\cite{FLOYD2007SUBFROMRAT}. Connections at vertices are investigated to develop a pseudo-isotopy $H^0$ which deforms $\gamma^0$ to $\gamma^1=F^{-1}(\gamma^0)$ relative to post. See \cite[Sec. 6-8]{MEYER2013INVARIANT}. 
\item The connections at vertices are  marked to assist with appropriately parameterizing $\gamma$ in a later step. See \cite[Defn 6.4]{MEYER2013INVARIANT}.
\item Successive lifts via $F$ of the pseudo-isotopy $H^0$ are used to develop a sequence of closed curves, $\gamma^n$. This sequence converges uniformly (with respect to the visual metric for $F$) to a curve $\gamma$. See Sections 3-4 of \cite{MEYER2013INVARIANT}.
\item A re-parameterization $\gamma^0:\mathbb{S}^1\rightarrow \mathbb{S}^2$ is assigned to $\gamma^0$, so that $F$ is topologically semiconjugate to the $dn$-fold map on $\mathbb{S}^1$ via $\gamma^0$. With this assigned parameterization, $\gamma^0$ is an oriented curve which is fully $F$-invariant up to homotopy. The successive lifts of $\gamma^0$ induce a parameterization of $\gamma$, which allows it to serve as an S1-parameterization of $\mathbb{S}^2$. See \cite[Sec 4]{MEYER2013INVARIANT}, as well as \cite[Sec. 9]{MEYER2014UNMATING} for implementation. %%%%
\item We consider an equivalence relation $\sim$ on $\mathbb{S}^1$ given by  $s\sim t \iff \gamma(s)=\gamma(t)$, along with several other equivalence relations induced by  $\gamma^0$. The equivalence relation $\sim$ can be interpreted in terms of limiting operations on the  $\sim_n$ relations, as well as in terms of a two-sided lamination. See \cite[Sec. 5-7]{MEYER2009EXPANDING}.
\item The two ``sides" of this lamination are generated by polynomials. Further, we can use the finite data from self intersections of $\gamma^1$ to generate critical portraits of the specific polynomials in the mating. See \cite[Sec. 7]{MEYER2009EXPANDING}, as well as \cite[Sec. 9]{MEYER2014UNMATING} for implementation. %CITE POIRIER?

\end{enumerate}

While Kameyama and Meyer's results are very similar, there are a few key differences that are of interest to us. First, the unmating process is much more explicitly constructive than Kameyama's--to the tune of well over a hundred pages of additional detail! This comes at the cost that the initial setting considered by Meyer is more restrictive: the oriented fully $f$-invariant up to homotopy curve that Meyer initializes with is Jordan. Indeed, when a map possesses a \emph{pseudo-equator} (roughly speaking, a Jordan $\gamma^0$ and the pseudo-isotopy noted in step 2 above), it is a sufficient condition for $F$ to be a mating--but Meyer notes that this criterion is not a necessary one. For example, any topological mating with no postcritical-to-postcritical identifications has a pseudo-equator \cite{WILKERSON2016FSR}, but there are examples in \cite{MEYER2013INVARIANT} and \cite{MEYER2014UNMATING} of matings that do not have one. Of note is that both of these examples \emph{do} still possess curves which are fully $f$ invariant up to homotopy.  Kameyama's construction, which effectively take the equator curve in a degenerate mating (as in Proposition 6.6 of \cite{KAMEYAMA2003PART2})--does not necessitate that $\mathcal{C}$ is Jordan. In the event that $\sim$ contains at least one equivalence class with two postcritical points from the same polynomial, the equator curve is ``pinched" in the resulting quotient space for the degenerate mating, and is no longer a Jordan curve. 

\subsection{An example}
We consider the case of the example from \cite[Sec. 10]{MEYER2013INVARIANT}:

Combinatorial data for the two polynomials and their mating are noted in Figures \ref{polynomials} and \ref{FSRisamating}. In \ref{polynomials}, $f_{5/12}$ and $f_{1/2}$ reference two Misiurewicz polynomials that are parametrized by angles of external rays landing at the respective critical values for each polynomial. The critical orbit portraits are as follows:

$f_{5/12}$:
\begin{tikzcd}
c_1 \ar[r,"2"]  & p_1\ar[r] & p_3  \ar[r] & p_0\ar[loop] 
\end{tikzcd}

$f_{1/12}$:
\begin{tikzcd}
c_2 \ar[r,"2"]  & p_2^*\ar[r] & p_3^*  \ar[r] & p_0^*\ar[r] & p_0' \ar[l, bend right] 
\end{tikzcd}

\begin{figure}[htb]
\center{\includegraphics[width=5in]{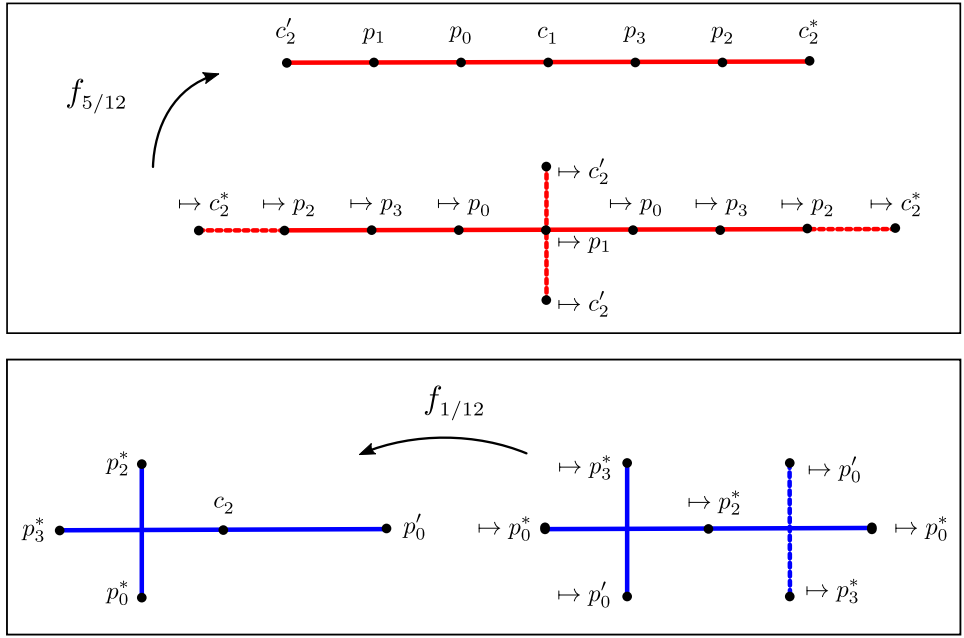}}
\caption{Combinatorial mapping data for two polynomials whose mating has no pseudo-equator. The Hubbard trees for $f_{5/12}$ and  $f_{1/12}$ are subsets of the above trees; additional edges and vertices have been marked to emphasize where postcritical identification will occur in the essential mating of these polynomials..} 
\label{polynomials}
\end{figure}

The graph structures noted in Figure \ref{polynomials} contain Hubbard trees for the polynomials $f_{5/12}$ and $f_{1/12}$, which give a combinatorial description of the mapping behavior of these polynomials. The postcritical points here are labeled to emphasize postcritical identifications forced by $\sim_e$: any points sharing a letter and subscript are collapsed in the topological and essential matings. We then default to using the relevant letter and subscript (sans other markings) to demonstrate the mapping behavior of the essential mating in Figure \ref{FSRisamating}.

Of note in Figure \ref{FSRisamating}: the dashed lines indicate a curve isotopic to $\mathbb{S}^1/\sim_e$ where $\mathbb{S}^1$ is the equator of the formal mating; along with its pullback. This is clearly no longer a Jordan curve, but there is a parameterization for the pullback of this curve which is orientation preserving and  pseudo-isotopic to the original relative to the postcritical set. (The dashed lines have been drawn slightly offset from the postcritical set to emphasize this.)

\begin{figure}[hbt]
\center{\includegraphics{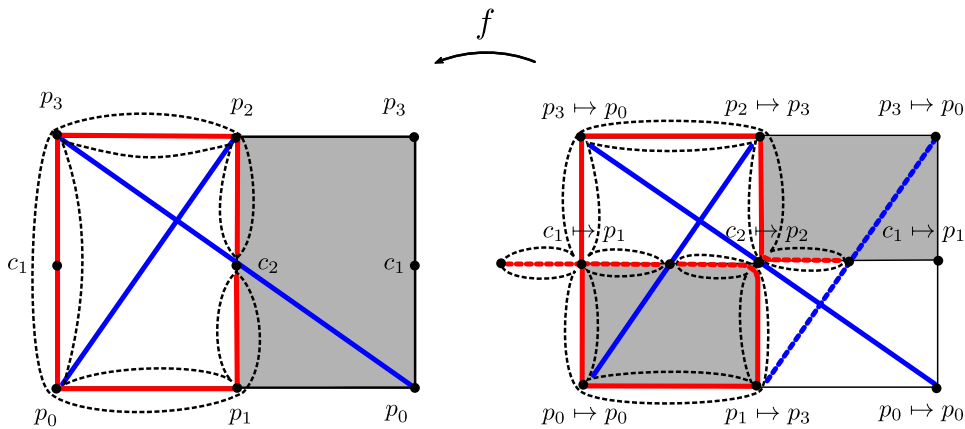}}
\caption{Meyer's example of a map with no pseudo-equator. The finite subdivision rule has been overlaid with 1-skeletons containing Hubbard trees for polynomials $f_{1/12}$ and $f_{5/12}$. This map is Thurston-equivalent to the essential mating of these polynomials.} 
\label{FSRisamating}
\end{figure}

%%%%%%%%%%%%%%%%%%%%%%%%%%%%%%%%%%%%%%%%                                                                                                                                                                                                                                                           %															%
%                                                             SECTION 6: UNMATING revised		%
%                                                             								%                                                                                                                                                                                                                                                           %															%
%%%%%%%%%%%%%%%%%%%%%%%%%%%%%%%%%%%%%%%%     
\section{Revising the Unmating Construction}\label{revised}

Let  $f$ be an expanding Thurston map with no periodic critical points, so that $J$ is the sphere.

An early step in Meyer's unmating process is the construction of a pseudo-isotopy which deforms a Jordan curve $\mathcal{C}$ containing post into its pullback by an iterate of $f$. We present an outline of Meyer's unmating algorithm with considerations for changes to be made in light of initializing with a fully $f$-invariant up to homotopy curve which is not necessarily simple. We begin by highlighting an immediate point of concern, and using this as a tool to discuss our approach through the rest of this section.

One potential problem to reconcile is that the Jordan curve $\mathcal{C}$ is used to generate much of the topological structure on $\mathbb{S}^2$ in which we discuss notions of convergence. However, as noted earlier in subsection \ref{expanding},  the visual metric $\rho$ and expansion constant $\Lambda$ that are associated with this $\mathcal{C}$ are not constrained to only work with $\mathcal{C}$: any Jordan curve through post can be shown to have a similar relationship (up to a scaling factor) with $\rho$ and $\Lambda$ \cite[Prop 8.3]{BONK2017EXPANDING}. With this in mind, many arguments posed involving open neighborhoods still apply because these sets are still open in our setting; however they may not relate to the surrounding structures in the initially intended manner. For example, consider the following:

\begin{lemma}\label{existence} (existence of neighborhoods between balls \cite[Lemma 2.3]{MEYER2013INVARIANT})
Let $\varrho$ be a visual metric for $f$ with expansion factor $\lambda$. Then there exist $\varepsilon_0>0$ and some constant $K\geq1$ such that the following is true: 

Suppose $0<\varepsilon<\varepsilon_0$ and let $\mathcal{N}_\varrho(\textbf{V}^1,\varepsilon)$ be the $\varepsilon$-neighborhood of $\textbf{V}^1$. Then there exists some neighborhood $V^1$ of $\textbf{V}^1$ such that $\mathcal{N}_\varrho(\textbf{V}^1, \frac{\varepsilon}{K}) \subset V^1\subset \mathcal{N}_\varrho( \textbf{V}^1 ,\varepsilon )$.

Further, the set $V^{n+1}:=F^{-n}(V^ 1)$ satisfies 

$\mathcal{N}_\varrho(\textbf{V}^{n+1},\Lambda^{-n}\frac{\varepsilon}{K})\subset V^{n+1}\subset \mathcal{N}_\varrho(\textbf{V}^{n+1},\Lambda^{-n}\varepsilon)$ for all $n\in\mathbb{N}$.

\end{lemma}

The above lemma is significantly dependent on \cite[Lemma 8.10]{BONK2017EXPANDING}, which we have reproduced here as Lemma \ref{open}. The $U^n(x)$ set referenced in Lemma \ref{open} is a union of $n$-tiles associated with a Jordan curve $\mathcal{C}$ through post. In Meyer's setting, the invariant curve $\gamma^0$ was being used with an iterate of $f$ for this $\mathcal{C}$, and this curve serves as the 1-skeleton for a 2-tile finite subdivision rule on $\mathbb{S}$. In other words, the $n$-tiles subdivide into $n+1$ tiles and the open set int $U^n(x)$ can be expressed neatly in terms of $n$-tiles of $\gamma^0$.

In our setting, we work with $f$ directly, and there may not even be a Jordan curve through post which is invariant with this map. We do assume, however that a Jordan $\mathcal{C}$ is chosen, and that the relevant $\varrho$ and expansion constant $\Lambda$ have been selected for our setting. Even though $\mathcal{C}$ doesn't generate a finite subdivision rule, we can still pull back the curve to generate $n$-tiles--so, $U^n(x)$ is still a meaningful object in this setting. However, our algorithm doesn't use $\mathcal{C}$ or any other Jordan curve as $\gamma^0$. The set $U^n(x)$ then will have nothing to do with our choice of $\gamma^0$ or its $n$-tiles, aside from living in the same metric space generated by $\rho$.

With that said, Lemma \ref{open} is only needed to obtain an open set of a particular size for Lemma \ref{existence}, and we do not in this instance require the combinatorial constraints of a finite subdivision rule--so existing in the same metric space is enough. Regardless of setting, int $U^n(x)$ can still function the desired open set, and so Lemma \ref{existence} still holds. 

We largely follow the constructions of Meyer outlined in the previous section, but for brevity assume that minor adjustments like the above are taken.  Our emphasis will be on places where the analagous constructions differ more significantly--however arguments and supporting lemmas will be summarized or reproduced to provide context. To assist the reader have tried to stay as close to notational conventions used in these papers as possible. Subsections here roughly align with the steps enumerated in the previous section. 

\subsection{Initializing with non-Jordan curves}
We begin by stating a key lemma from Meyer's construction to initialize our unmating.

\begin{lemma}\label{schonflies} (Isotopic Sch\"{o}nflies theorem \cite[Thm 5.1]{MEYER2013INVARIANT}) Let $\gamma, \sigma\subset \mathbb{D}$ be two Jordan arcs with common endpoints $p,q\in\overline{\mathbb{D}}$. Then, there is an isotopy of $\overline{\mathbb{D}}$ relative to $\partial \mathbb{D}\cup\{p,q\}$ that deforms $\gamma$ to $\sigma$. 
\end{lemma}

\begin{theorem}\label{OG1}
Suppose that $f$ is an expanding degree $d$ Thurston map whose Julia set is $\mathbb{S}^2$, and let $\mathcal{\gamma^0}:\mathbb{S}^1\rightarrow\mathbb{S}^2$ be an oriented closed curve which is fully $f$-invariant up to homotopy. Then, there exists a pseudo-isotopy $H_0: \mathbb{S}^2 \times [0,1]\rightarrow \mathbb{S}^2$ satisfying the following properties:

(1) $H^0$ is a pseudo-isotopy relative to $V_0=$post.

(2) The set of all 0-edges is deformed by $H^0$ to the set of all 1-edges.  

(3) Only finitely many points of $\gamma^0$ are deformed by $H_0$ to any individual 1-vertex v. 

(4) Let $\varepsilon_0 >0$ be the constant from Lemma \ref{existence}, $0<\varepsilon<min\{\varepsilon_0, 1/2\}$ and $V^1$ be a neighborhood of $\textbf{V}^1$ also as in Lemma \ref{existence}. We will further require that $H^0: \mathbb{S}^2 \times[1-\varepsilon, 1]\rightarrow\mathbb{S}^2$ is supported on $V^1$. 
\end{theorem}

\begin{proof}
Without loss of generality, we may assume that $\gamma^0$ has at most finitely many self-intersections occurring only at points in post \cite[Thm 3.6]{KAMEYAMA2003PART2}. As such, we may assume that $\gamma^1=f^{-1}(\gamma^0)$ is an oriented closed curve with at most finitely many self-intersections as well, and construct $H^0$ as follows.

Since $\gamma^0$ is oriented, a small perturbation of this curve yields a Jordan curve $\gamma^{0'}$ which differs from $\gamma^0$ only on small neighborhoods surrounding its points of self-intersection. We obtain a similar result for $\gamma^1$, where a small perturbation on neighborhoods $U_1,...,U_n$ of intersection points yields the Jordan curve $\gamma^{1'}$. We may take $\gamma^{1'}$ to be an ``unlacing" of $\gamma^1$ which is homeomorphic to $\gamma^{0'}$ and which traverses the elements of post in the same cyclic order. Repeated application of Lemma \ref{schonflies} to corresponding boundary arcs of $\gamma^{0'}$ and $\gamma^{1'}$ yields an isotopy $H^a:\mathbb{S}^2\times[0,1-\varepsilon]\rightarrow \mathbb{S}^2$ that deforms $\gamma^{0'}$ into $\gamma^{1'}$. We may concatenate this isotopy with a pseudo-isotopy $H^b:\mathbb{S}^2\times[1-\varepsilon]\rightarrow\mathbb{S}^2$ that fixes post, but deforms the Jordan curve $\gamma^{1'}$ into $\gamma^1$. We will call this concatenation the pseudo-isotopy $H^0$.

Criteria (1) and (2) follow immediately from the construction. Note that the cycle of postcritical points visited by $\gamma^0$ is a subsequence of those visited by $\gamma^1$, which includes repeat visits to elements of post that are included at the very end of the pseudo-isotopy $H^0$. As  Lemma \ref{schonflies} is the machinery behind the pseudo-isotopy $H^b$, (3) follows. We may guarantee (4) with careful selection of both $\varepsilon$ and  a maximum allowable diameter for the sets $U_1,...,U_n$. 
\end{proof}

\begin{remark} It should be noted that the original construction of $H^0$ requires a 5th property regarding $f:H^0_1(\gamma^0)\rightarrow\gamma^0$ being a $d$-fold cover, however we start with a stronger condition in light of $\gamma^0$ being fully $f$-invariant up to homotopy: this implies by definition that $f\circ \gamma^1=\gamma^0 \circ q_d$. \end{remark} 

\begin{remark}Per the proof of \cite[Lemma 3.11]{MEYER2013INVARIANT}, we expect this $d$-fold cover property exactly when $f$ is orientation preserving. In light of this, applying the above construction with a Jordan $\gamma^0$ would satisfy conditions for what Meyer calls a \emph{pseudo-equator.} \end{remark}

\begin{remark}If a map has a pseudo-equator it is equivalent to a mating, but at the comparable point in Meyer's construction $F=f^n$ is used to pull back $\gamma^0$ to obtain $\gamma^1$, where we instead use $f$. This demonstrates that the iterate $f^n$ is a mating instead of $f$.
\end{remark}

%%%%%%%%%%%%%%%%
\subsection{Marking connections at vertices}
%%%%%%%%%%%%%%%%%%

As the $H^0$ of Theorem \ref{OG1} initially assumes $\gamma^0$ to be fully $f$-invariant up to homotopy, its relationship with $q_d$ yields that our parameterizations for $\gamma^0$ and $\gamma^1$ are implied. With this extra information, we have less to construct than in \cite{MEYER2013INVARIANT}: for Meyer there may be multiple oriented Eulerian paths along the set of points in $F^{-1}(\gamma^0)$--and one must be constructed and selected before constructing any sort of pseudo-isotopy.  Meyer's tile-centered construction thus eventually allows for the discovery of shared matings--however, these typically reflect an iterate of $f$ rather than $f$ itself.  

With that said, as long as the Eulerian paths determined by $\gamma^0$ and $\gamma^1$ are understood up to homotopy type, we can recover an S1-parameterization for $f$ via the methods described here with or without the assumed parameterizations. We just need to know the sequence of postcritical points visited along $\gamma^0$ (in order, and including return visits), can be viewed as a subsequence of the  postcritical points visited in order along the pseudo-isotopic pullback curve $\gamma^1$.

To do this requires an understanding of connections--and we must be particularly careful because $\gamma^0$ is not necessarily assumed to be Jordan. This means we need to note connections of both $\gamma^0$ \emph{and} $\gamma^1$ instead of just $\gamma^1.$ To help with this, we will adopt a new convention regarding \emph{marking} vertices at connections, which is singling out an arc on a curve in order to ``keep track" of where we expect a relevant postcritical point to be. (In figures of a geometric representation, Meyer typically denotes this with a dot on the relevant arc; we may name or label vertices and/or arcs as needed.) Meyer marks a single arc at each connection in his version of $\gamma^1$; we mark \emph{all} arcs in the geometric representation of $\gamma^0$. As elements of post are possible intersection points of $\gamma^0$, we  expect that multiple parameters may yield a visit to the same point--and we ``keep track" of the parameters accordingly. 

Note that considerations determining which vertices are connected and how markings are constructed, etc. doesn't really differ from in the unmating approach---we simply have more markings to keep track of. 

%%%%%%%%%%%%%%%%
\subsection{The limit of approximating curves $\gamma^n$ }
%%%%%%%%%%%%%%%%%%

We now develop a sequence of approximations to the desired S1-parameterization by taking lifts of our pseudo-isotopy. To do so, we recall a fundamental result for lifts of pseudo-isotopies:

\begin{lemma}\label{douevenliftbro} (Lifts of pseudo-isotopies \cite[Lemma 3.4]{MEYER2013INVARIANT}) Let $H: \mathbb{S}^2 \times[0, 1]\rightarrow\mathbb{S}^2$ be a pseudo-isotopy relative to post$=\textbf{V}^0$. Then $H$ can be lifted uniquely by $f$ to a pseudo-isotopy $\tilde{H}_1$ relative to $f^{-1}(\textbf{V}^0)=\textbf{V}^1$. This means that $f(\tilde{H}_1(x,t))=H(f(x),t)$ for all $x\in\mathbb{S}^2$ and all $t\in[0,1]$.

Further, let $H_n$ be the lift of $H_0$ by an iterate $F^n$. Then diam $H_n :=$ max $_{x\in\mathbb{S}^2}$ diam$\{H_n(x, t) : t \in [0, 1]\}$, and we have that diam $H_n\leq C\Lambda^{-n}$.

(Diameter here is measured with respect to the fixed visual metric $\varrho$ with expansion factor $\Lambda>1$. The constant $C$ is independent of $n$.) 
\end{lemma}

The argument supporting this lemma is independent of the choice of curve $\mathcal{C}$ that $\varrho$ and $\Lambda$ are associated with, and also independent of the initial pseudoisotopy $H$. Using the $H^0$ of section \ref{OG1}, we may now repeatedly apply Lemma \ref{douevenliftbro} to obtain a sequence of unique pseudo-isotopies $\{H^n$ relative to $\textbf{V}^n\}$ deforming $\gamma^{n}$ into $\gamma^{n+1}$. 

We obtain the following analogue of \cite[Lemma 3.5]{MEYER2013INVARIANT}, which states how the properties of $H^0$ noted in Theorem \ref{OG1} convey to the pseudo-isotopies $\{H^n$ relative to $\textbf{V}^n\}$.

\begin{lemma}
Let $H^0$ be a pseudo-isotopy as in Theorem \ref{OG1}, and $H^n$ denote the lift of $H^0$ by $f^n$. The lifts $H^n$ possess the following properties:

(1) $H^0$ is a pseudo-isotopy relative to $\textbf{V}^n=f^{-n}(\text{post})$.

(2) The set of all $n$-edges is deformed by $H^n$ to the set of all $(n+1)$-edges.  

(3) Only finitely many points of $\gamma^n$ are deformed by $H^n$ to any individual $(n+1)$-vertex $v$. 

(4) Let $V_{\varepsilon}^1$ be the neighborhood of $\textbf{V}^1$ as in Theorem \ref{OG1}. Then $H^n$ is supported on the set $\textbf{V}^{n+1}:=f^{-n}(\textbf{V}^1)$, which is a neighborhood of $\textbf{V}^{n+1}$.

\end{lemma}

\begin{proof}
This largely follows in an identical manner to the analogous Lemma in \cite[Lemma 3.5]{MEYER2013INVARIANT}. Statement (2) requires a small amount of additional bookkeeping since it is possible for more than one 0-edge in our setting to start at the same point: Meyer uses the fact that lifts of paths with distinct starting points have distinct endpoints, and uses this to count lifts of edges. We cannot precisely replicate this argument as some distinct 0-edges share initial points. Sub-arcs contained in these 0-edges could be constructed to have unique initial points though--and so we can still appropriately index the necessary lifts to guarantee an appropriate number of $n+1$ edges. 
\end{proof}

%%%%%%%%%%%%%%%%
\subsection{Parameterization of $\gamma$}
%%%%%%%%%%%%%%%

In this subsection we outline the procedures in \cite[Sec. 4]{MEYER2013INVARIANT}, \cite[Sec. 9]{MEYER2014UNMATING} for context, as they are largely unaltered in our setting.

Note that if we know which values of $t$ yield the points in $\textbf{V}^n$ that we may pull back to obtain the parameters for the elements of $\textbf{V}^{n+1}$, and so on as these have been appropriately marked on $\gamma^{n+1}$. Taking the $\gamma^n$ as any orientation-preserving homeomorphism on the remaining arcs, we have that $\gamma: \mathbb{S}^1\rightarrow \mathbb{S}^2$ yields an oriented S1-parameterization of $\mathbb{S}^2$ when defined so that $\gamma(t):=\displaystyle\lim_{n\rightarrow\infty}\gamma^n(t)$. 

The beauty in the argument is that we obtain this parameterization starting with finite data—we need only start with parameters for the elements in post. The parameters for $\textbf{V}^0=$post separate $\mathbb{S}^1$ into a collection of subarcs. Since we expect $f\circ\gamma(t) =\gamma(d\cdot t)$, we know which subarcs map onto others; and can develop an edge transition matrix to represent this transition. An extension of the Perron-Frobenius theorem guarantees a unique positive eigenvalue equal to the spectral radius of this matrix, which in this setting is the degree of $f$. This eigenvalue is simple, and has an eigenvector with all positive values, representing respective lengths of each of the arcs on $\mathbb{S}^1$ (after a suitable scaling for overall unit length). 

Knowing the lengths of the respective arcs now means that given a parameter $t$ for any postcritical point $p$, we have two ways of obtaining the parameter for $f(\gamma(t))$: we can add the appropriate lengths of arcs to $t$ to find the new parameter, or we can multiply $t$ by $d$ to obtain the new parameter. This allows us to set up an equation to solve for $t$.

With the parameterizations above for $\gamma$ and each of the $\gamma^n$, we have that the sequence of curves $\{\gamma^n\}$ converge uniformly to $\gamma$ in the visual metric $\varrho$. Further, if desired, we could construct a pseudo-isotopy from $\gamma^0$ to $\gamma$ by modifying a concatenation of the sequence of homotopies $\{H^n\}$. 

\begin{remark} In our setting, we begin with $\gamma^0$ an oriented curve which is fully f-invariant; implying a given parameterization. However, we did not need to actually know this parameterization and could recover the parameters for the postcritical points without it using the above procedure. The only crucial difference in application is that with a self-intersecting $\gamma^0$, postcritical points may have been represented with multiple parameters—however we have offset this by our choice to make additional markings on $\gamma^0$ to compensate. The application of techniques in this section is otherwise identical.
\end{remark}

%%%%%%%%%%%%%%
\subsection{Laminations and equivalence relations}
%%%%%%%%%%%%%%

At this point, $\gamma$ and the sequence of $\gamma^n$ are used to generate several equivalence relations on $\mathbb{S}^1$:

\begin{itemize}
\item The equivalence relations  $\overset{_{n,w}}{\sim}$ and $\overset{_{n,b}}{\sim}$: Consider a self intersection point $v$ of $\gamma^n$. Let $\gamma^{n'}$ be an unlacing of $\gamma^n$, and $U(v)$ be a small neighborhood of $v$ as in Definition \ref{mark}. We now mark each arc of $U(v)\cap\gamma^{n'}$, assuming that these marks correspond to angles in $\mathbb{S}^1$ which $\gamma^n$ sends to $v$.  If any of the marked points lie on the boundary of the same white component of $U(v)-\gamma$, we say that their corresponding angles are in the same equivalence class of the equivalence relation $\overset{_{n,w}}{\sim}$. (We may similarly define an equivalence relation for points lying on the boundary of the same black component, $\overset{_{n,w}}{\sim}$.)

\item The closure of $\overset{_\infty}{\sim}$: We define the equivalence relations $\overset{_n}{\sim}$ so that $s\overset{_n}{\sim}$ if and only if $\gamma^n(s) = \gamma^n(t)$. (As the $\gamma^n$ are inductively defined, in a sense these can be interpreted as inductively defined from $\overset{_0}{\sim}$.) We then take $\overset{_\infty}{\sim}:=\bigvee\overset{_n}{\sim}$. The closure of this is an equivalence relation.

\item The equivalence relation $\sim$: We define $\sim$ so that $s\sim t$ if and only if $\gamma(s) = \gamma(t)$.

\item The equivalence relation $\overset{_w}{\sim} \bigvee \overset{_b}{\sim}$: We proceed similarly to above. Note that we have already defined the equivalence relations $\overset{_{n,w}}{\sim}$ and $\overset{_{n,b}}{\sim}$, and that these similarly can be interpreted as inductively defined from $\overset{_{0,w}}{\sim}$ and $\overset{_{0,b}}{\sim}$ respectively. We take $\overset{_w}{\sim}\ :=\ \bigvee\overset{_{n,w}}{\sim}$, and similar for $\overset{_{b}}{\sim}$. $\overset{_w}{\sim} \bigvee \overset{_b}{\sim}$ is an equivalence relation.

\item The equivalence relation $\approx\ :=\ \overset{_w}{\approx}\bigvee\overset{_b}{\approx}$: The equivalence relation $\overset{_{1,w}}{\sim}$ is taken to generate a \emph{critical portrait} for the white polynomial $P_w$. We take $\overset{_w}{\approx}$ to be the equivalence relation induced by the Carath\'{e}odory semiconjugacy $\sigma_w$, for $P_w$ which we discussed in Section \ref{matings}. In other words,  $s\sim t$ if and only if $\sigma_w(s) = \sigma_w(t)$. We similarly define $\overset{_b}{\approx}$.

\end{itemize}

Aside from a differing $\overset{_{0,w}}{\sim}$ and $\overset{_{0,b}}{\sim}$ in our setting, we take our equivalence relations to be defined and/or induced in the same respective manner.

In the setting of \cite{MEYER2009EXPANDING}, all of these equivalence relations (except for the $\overset{_{n,w}}{\sim}$ and $\overset{_{n,b}}{\sim}$) are noted to be equal. We first have that $\overset{_\infty}{\sim}=\sim$ as a result of a metric argument in Theorem 4.7.  We have that  $\sim=\overset{_w}{\sim} \bigvee \overset{_b}{\sim}$ via a similar metric argument along with an exercise in operations on equivalence relations. Proposition 7.12 guarantees that $\overset{_w}{\sim} \bigvee \overset{_b}{\sim}=\approx$ by noting that $\overset{_w}{\sim} =\overset{_w}{\approx}$ via an argument regarding the polynomial $P_w$ and Poirier's Theorem, and similar for the corresponding black equivalence relations. Further, $z^d/\approx$ is noted to generate a map topologically conjugate to the topological mating. In other words, this links Meyer's S1-parameterization to the topological mating of the two polynomials $P_w$ and $P_b$. 

Many of the arguments here rely on discussion of the underlying metric, and/or are exercises in properties of operations on equivalence relations. Since we are in the same metric space and have defined our equivalence relations similarly, these arguments do not change in a significant manner. The primary point of concern is that our $\overset{_{0,w}}{\sim}$ and $\overset{_{0,b}}{\sim}$, which induce the rest of these equivalence relations, are defined differently than in \cite{MEYER2009EXPANDING}. Meyer uses the nontrivial equivalence classes of $\overset{_{1,w}}{\sim}$ to generate a \emph{critical portrait} which yields the polynomial $P_w$; this is a critical step in defining $\overset{_w}{\approx}$ and asserting that this equivalence relation is equal to $\overset{_w}{\sim} $.

It thus remains to discuss what a \emph{critical portrait} is, and to show that our $\overset{_{1,w}}{\sim}$ and $\overset{_{1,b}}{\sim}$  relations can be used to define them.

%%%%%%%%%%%%%%%%
\subsection{Critical portraits}
%%%%%%%%%%%%%%%%

We define the following per \cite{poirier2009critical}:

\begin{definition}
We define that a set $A\subset \mathbb{S}^1$ is a \emph{degree $d$ preargument set} if $da = \{da : a \in A\}$ is a singleton set.

Further, consider a pair of families $\mathcal{F}= \{\mathcal{F}_1, . . . , \mathcal{F}_n\}$ (associated with critically periodic orbits) and $\mathcal{J} = \{\mathcal{J}_1, . . . , \mathcal{J}_m\}$ (associated with critically preperiodic orbits) of rational degree $d$ preargument sets. We say that $\Theta = (\mathcal{F} , \mathcal{J} )$ is a \emph{degree $d$ critical portrait} if the following conditions are fulfilled:

\begin{itemize}
\item[(c1)] $d-1=\displaystyle\sum(|\mathcal{F}_k|-1)+\sum(|\mathcal{J}_k|-1),$
\item[(c2)] $\mathcal{J}$ is weakly unlinked to $\mathcal{F}$ on the right,
\item[(c3)] each family is hierarchic, (That is, whenever a degree $d$ preargument set is mapped into from another preargument set via some $a\mapsto d^ia\in A$--then $d^ia\in A$ is the ``preferred" or representative element of $A$ that is mapped to by any preargument set.)
\item[(c4)] for any $a$ that participates in $\mathcal{F}$, some periodic forward iterate $d^ia$
also participates in $\mathcal{F}$
\item[(c5)] no $a$ that participates in $\mathcal{J}$ is periodic.
\item[(c6)] Suppose $a$ and $t$ are periodic with the same period and the same symbol sequences. If $a$ participates in $\mathcal{F}$, then $a = t$.
\item[(c7)]Let $a\in\mathcal{J}_l$ and $t \in\mathcal{J}_k$.Take $i\geq0$.If the left symbol sequence of $d^ia$ equals the left symbol sequence of $t$ then $d^ia = t$.
\end{itemize}
\end{definition}

Poirier's Theorem notes that critical portraits determine centered monic postcritically finite polynomials.

Meyer’s choice of Jordan $\gamma^0$ yields that any point of intersection on $\gamma^1$ is automatically a critical point, since a point of intersection does not map homomorphically to any arc of $\gamma^0$. The parameters associated with that point automatically generate appropriate preargument sets, which are used to form critical portraits for the relevant polynomials.

We cannot simply choose any point of intersection in $\gamma^1$ to yield a preimage set associated with a critical point though, as some of our intersection points on $\gamma^1$ may map homeomorphically onto intersection points of $\gamma^0$. Further, as some critical values are marked by multiple parameters, we need to ensure that we actually obtain proper preargument sets that do not contain too many elements.  Therefore:

\begin{definition}
We form markings of select vertices of $\textbf{V}^1$ as follows: 

\begin{itemize}

\item Note the elements of $\textbf{V}^1$ which do not map homeomorphically onto vertices of $\textbf{V}^0=post$. These vertices $c_1,…,c_k$ form our critical set.

\item Consider an unlacing $\gamma^{1'}$ of $\gamma^{1}$ in a small neighborhood $U(c)$ of each critical point $c$, and examine the components of $U(c)-\gamma^{1'}$. If any component does not map homeomorphically onto its image, note the color of this component, and partition the critical points by color.

\item Start with the white critical points. Select a critical point (again, call it c for now) that is not in the forward orbit of any other critical point of the same color. Select any marked parameter $\alpha$ for which $\gamma(\alpha)=c$. The marking for this critical point is now the collection of values in the set $q_d^{-1}(q_d(\alpha))$.

\item Iterate $\alpha$ in $q_d$ until we find a parameter for the next white critical point in the forward orbit. Apply $q_d^{-1}\circ q_d$ to this critical parameter to obtain a marking for this new critical point. Repeat this process to continue marking criticals point until we have exhausted critical points in the forward orbit of $c$.

\item If there are remaining critical points that eventually map into this critical orbit, select the parameter associated with this critical point that eventually iterates in $q_d$ to a previously noted parameter. Similarly apply $q_d^{-1}\circ q_d$ to this critical parameter to find a set of values to mark the critical point.

\item Repeat this process until all white critical points are exhausted. Then, apply this method to the set of black critical points. 
\end{itemize}

The collection of white markings (or preimage sets) yields the \emph{critical portrait} for the white polynomial $P_w$, and similar for the black polynomial $P_b$.
\end{definition}

\begin{theorem}
The above procedure generates a critical portrait for both the black and white polynomials. 
\end{theorem}

\begin{proof}
As our case involves preperiodic critical points only; the criteria for a degree $d$ preargument set satisfying the conditions of a critical portrait are significantly simplified: any condition involving $\mathcal{F}$ is trivially satisfied because $\mathcal{F}=\varnothing$ here. The considerations for Rieman-Hurwitz are satisfied by an argument similar to the one that Meyer presents in the proof of \cite[Lemma 5.7]{MEYER2009EXPANDING}. That this serves as a critical portrait follows mostly from construction; the one-by-one development of critical markings is intended to select preferred elements to preserve the hiearchic structure required of the remaining considerations for critical markings.
\end{proof}

\subsection{An example, revisited}
We revisit the function $f$ of our earlier example, this time demonstrating the steps noted above.\\

\noindent (1) We begin with an oriented curve $\gamma^0$ which is fully f-invariant up to homotopy. As noted previously, this implies an existing parameterization for $\gamma^0$ and its pullback $\gamma^1$, but the construction only necessitates an understanding of the homotopy types of these curves, as conveyed in Figure \ref{pullbackcurve1}. 

\begin{figure}[htb]
\center{\includegraphics{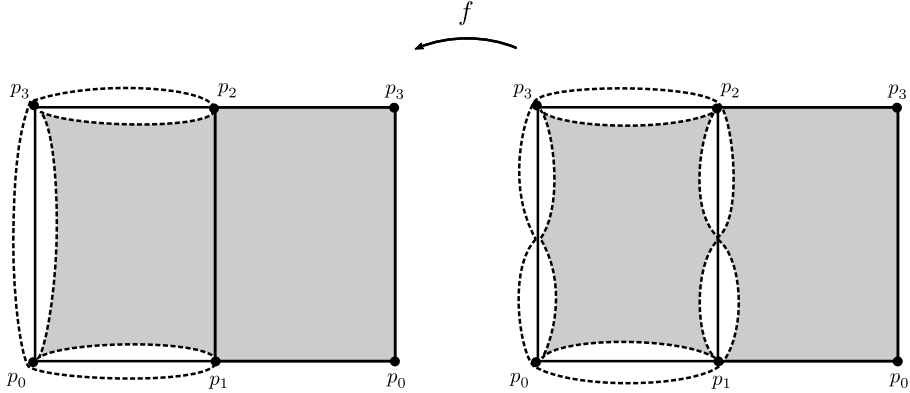}}
\caption{The dashed lines above note an oriented curve $\gamma^0$ with its marked postcritical set, as well as the curve $\gamma^1:=f^{-1}(\gamma^0)$.}
\label{pullbackcurve1}
\end{figure}

\noindent (2) We consider an unlacing of $\gamma^0$ and $\gamma^1$ and mark these curves accordingly, as in Figure \ref{pullbackcurve2}. Note that at each point of self-intersection on $\gamma^0$ all vertices are marked. Since $\gamma^0$ is fully $f$-invariant up to homotopy, we have that $f\circ \gamma^1=\gamma^0\circ q_2$, which induces a marking on $\gamma^1$. We construct the pseudo-isotopy $H^0$ of Theorem \ref{OG1} to respect this induced marking. 

\begin{figure}[htb]
\center{\includegraphics{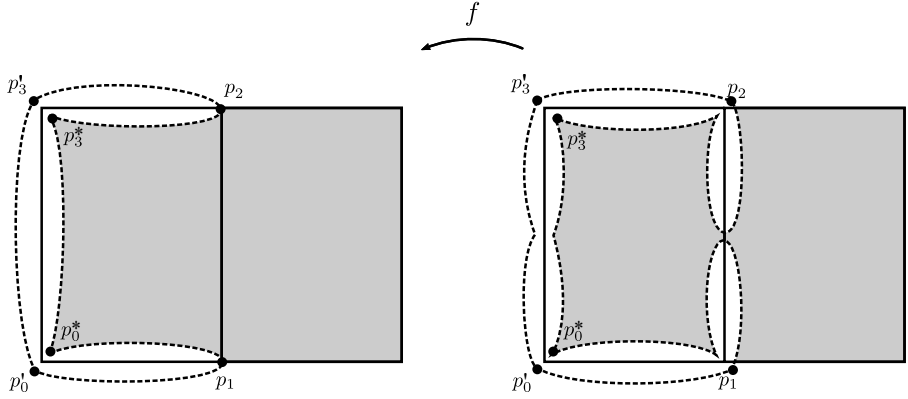}}
\caption{We select a marking of curves.} 
\label{pullbackcurve2}
\end{figure}

One may note that this is not the same oriented fully f-invariant curve of Kameyama's construction for the mating which is initially pictured in Figure \ref{FSRisamating}. This was done intentionally to highlight how the marking conventions allow us to ``keep track of postcritical points", as the later parameterizations of $\gamma^n$ fix values of $t$ at these locations. It should be noted that the procedure could be repeated with this curve to obtain the same polynomial decomposition, though. %and 2. how forgiving the algorithm is as long as the curve selected notes some of the more necessary intersections which are due to self-postcritical i HOLD THIS THOUGHT FOR DISCUSION.

\noindent (3) We may repeatedly lift pseudo-isotopies to obtain a sequence $H^n$ and corresponding collection of curves $\gamma^n$. We take there to be some limit curve $\gamma$ whose parameterization we develop in the next step. 

\noindent (4) The parameterization of $\gamma$ is induced by a finite amount of information namely, we need to recover the parameters at each element of post. To do so, we mark the edges of $\gamma^0$, and note locations of their pullbacks via $f$, as in Figure \ref{pullbackcurve3}. 

\begin{figure}[htb]
\center{\includegraphics{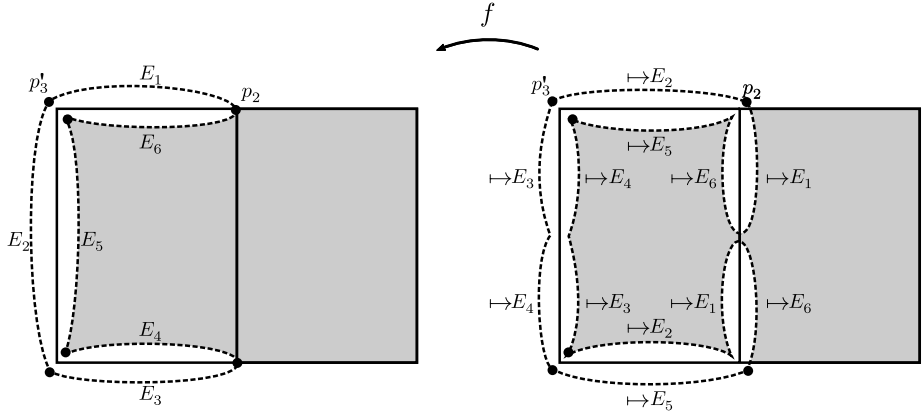}}
\caption{We label 0-edges $E_i$, as well as pullbacks. } 
\label{pullbackcurve3}
\end{figure}

We may then develop an edge transition matrix whose entries $a_{ij}$ are determined by counting how many 1-edges which map to $E_j$ are contained in the $H_1$ deformation of each edge $E_i$. For instance, we can see that $H^0$ deforms edge $E_6$ into an arc which contains sub-arcs mapping to $E_1, E_5,$ and $E_6$--and this deformation is responsible for each entry on the last row of the matrix below.

\begin{center}
$$\begin{bmatrix}
0&1&0&0&0&0\\
0&0&1&1&0&0\\
0&0&0&0&1&0\\
1&1&0&0&0&1\\
0&0&1&1&0&0\\
1&0&0&0&1&1\\

\end{bmatrix}$$
\end{center}

\vspace{.25in}

The spectral radius of this matrix matches the degree of $f$, which is 2. This eigenvalue corresponds to a unique real eigenvector, $[1\ 2\ 1\ 3\ 2\ 3]^T$. If we scale the entries of this vector so that the components sum to 1, each $i$th component $l_i$ represents the length of the interval $e_i\in\mathbb{S}^1$ which maps to $E_i$: For instance, the parameters mapping to $p_2$ and $f(p_2)=p_3'$ are $l_1=1/12$ units apart.

We can then use the lengths we've computed along with the mapping behavior of $f$ to compute parameters corresponding to each of the markings representing elements of post. Suppose $p_a,p_bf(p_a))\in$post are two postcritical points with corresponding parameters $t_a,t_b\in \mathbb{S}^1$. Then
 not only is $t_b=d\cdot t_a$, we also can find $t_b$ by adding the sum $l(t_a,t_b)$ of lengths of subintervals between $t_a$ and $t_b$ in $\mathbb{S}^1$.  
 
 We obtain that $d\cdot t_a=t_a+l(t_a,t_b)$ mod $1$, or that one choice for $t_a$ is $t_a=l(t_a,t_b)/(d-1)$. (Any other choices correspond to an isomorphism of the mating.) Applying this to $p_a=p_2$, we get that the parameter of $t$ corresponding to $p_2$ is 1/12. We can then successively add appropriate $l_i$ to this parameter to determine the successive parameters of marked elements of post along $\gamma^0$: e.g. 1/12 +$l_1$ gives the parameter for $p_3'$, adding $l_2$ to this gives the parameter for $p_0'$, and so on. We obtain that the parameters for our postcritical markings (starting at $p_2$ and traveling in a positive direction with respect to the white tile) are $\frac{1}{12}, \frac{1}{6}, \frac{1}{3}, \frac{5}{12}, \frac{2}{3}, $ and $\frac{5}{6}.$

\noindent (5) We do not picture the following pullbacks here, but the parameterization induced by the lifts of pseudo-isotopies generates an S1-parameterization of $\mathbb{S}^2$. As such, this map arises as a mating of two polynomials, $P_w,$ and $P_b$ whose critical portraits we find below.

\noindent (6) We note that the middle left ``pinched spot" on the white collection of $1-tiles$ maps non homeomorphically to the postcritical point $p_1$, and as such is one of our critical points. Applying $q_2^{-1}\circ 2_d$ to this parameter is the same as applying $q_2^{-1}(5/12)$ at $p_1$, which yields a critical portrait $P_w=\{\frac{5}{24}, \frac{17}{24}\}$. 

The similar construction for the black polynomial yields $P_b=\{\frac{1}{24},\frac{13}{24}\}$.

%%%%%%%%%%%%%%%%%%%%%%%%%%%%%%%%%%%%%%%%                                                                                                                                                                                                                                                           %															%
%                                                             CLOSING						%
%                                                             (implications for future study)		%                                                                                                                                                                                                                                                           %															%
%%%%%%%%%%%%%%%%%%%%%%%%%%%%%%%%%%%%%%%%   
\section{Future study}\label{conclusion}

While this approach aligns the unmating approach more with the cases suggested in \cite{KAMEYAMA2003PART2}, it does leave something to be desired that the initialization of the algorithm is not as constructive. The author has interest in whether further study of more general subdivision rules (rather than those of just 2-tilings) could yield a more constructive approach, or if subdivisions could be used to single out which elements of post are the ``problematic" ones requiring multiple markings. Alternatively, in the spirit of noting that Thurston's pullback algorithm is frequently known to converge despite obstructions (e.g. \cite{BOYD2012MEDUSA}), there may be stipulations for when it is possible to force a 2-tile decomposition procedure without resorting to iterates of $f$.

Another potential avenue of investigation for the techniques presented here regard what have sometimes been called \emph{anti-pseudo-equators}. In some cases, the orientation of the pullback curve $f^{-1}(\mathcal{C})$ is reverse that of $\mathcal{C}$, which is one reason for passing through to an iterate of $f$ before attempting an unmating. For instance, in Example 6.6 of \cite{MEYER2014UNMATING} the pullback curve traverses $\text{post}$ in reverse order from that of $\mathcal{C}$. It has been conjectured by Meyer and Jung \cite{JUNG2020ANTIMATINGS} that these may be a result of functions called anti-matings, and that it may be possible to extend the unmating algorithm in these cases. Notably, similar maps have been studied by Timorin \cite{TIMORIN2008EXTERNAL} and Dastjerdi \cite{dastjerdi1991dynamics} in terms of laminations, and unlike matings where the two ``sides" of the lamination are fixed; these are represented by laminations where the two sides swap on each iteration.

 %%%%%%%%%%%%%%%%%%%%%%%%%%%%%%%%%%%%%%%%                                                                                                                                                                                                                                                           %															%
%                                                             SECTION: ACKNOWLEDGEMENTS	%
%                                                             (put your thank you's here)			%                                                                                                                                                                                                                                                           %															%
%%%%%%%%%%%%%%%%%%%%%%%%%%%%%%%%%%%%%%%%   
 
\section*{Acknowledgements}
 
The author would like to thank Wolf Jung for his assistance in obtaining a difficult to find reference, as well as for a helpful review comment on a previous paper which prompted some of this investigation. She would also like to acknowledge the editor for his patience in handling this manuscript.

\bibliography{Wilkerson.bib}
\bibliographystyle{amsplain}
 
\end{document}